\documentclass[12pt,a4paper]{article}
\pdfoutput=1
\usepackage{amssymb}
\usepackage{amsthm}
\usepackage{amsfonts}
\usepackage{graphicx} 

 \newtheorem {theorem} {Theorem} 
 \newtheorem {proposition} {Proposition} 
 \newtheorem {lemma} {Lemma}
 \newtheorem {definition}{Definition}
  
 \def\R{\mathbb{R}}
 
 \def\H{\mathbb{H}^2_+}
  \def\S{\mathbb{S}^2_+}
 \advance \textheight by \topskip
\title {Periodic Orbits of Oval Billiards on Surfaces of Constant Curvature } 
\author {Luciano Coutinho dos Santos \footnote{CEFET-MG, Brazil, email: astrofisico2@deii.cefetmg.br} \quad\quad S\^onia Pinto-de-Carvalho\footnote{UFMG,Brazil, email: sonia@mat.ufmg.br}} 
\date {}

\begin{document}
\maketitle 

\begin{abstract}

In this paper we define and study the billiard problem on bounded regions on surfaces of constant curvature. We show that this problem defines a 2-dimensional conservative and reversible dynamical system, defined by a Twist diffeomorphism, if the boundary of the region is an oval. Using these properties and defining good perturbations for billiards, we show that having only a finite number of nondegenerate periodic orbits for each fixed period is an open property for billiards on surfaces of constant curvature and a dense one on the hyperbolic plane. We finish this paper studying the stability of these nondegenerate orbits.
\end{abstract}

\section{Introduction}

The planar billiard problem, originally defined by Birkhoff \cite{bir} in the beginning of the XX century, consists in the free motion of a point particle in a bounded planar region, reflecting elastically when it reaches the boundary. 

In this  work we extend this problem to bounded regions  on geodesically convex subsets of  surfaces of constant curvature. We will show that this new billiard also defines a 2-dimensional conservative and reversible dynamical system, defined by a Twist diffeomorphism, if the boundary of the region is an oval, i.e., a regular, simple, closed, oriented, strictly geodesically convex and at least $C^2$ curve. This is a classical result  for the Euclidean case, proved, for instance, in \cite{kat}.  

Once we have proved that we have a very special dynamical system, we address the question of how many $n$-periodic orbits such a billiard can have. Bolotin \cite{bol} proved that the geodesic circular billiard on surfaces of constant curvature is integrable and then has infinitely many  orbits of any period. A classical result for Twist maps (see, for instance, \cite{kat}), proved for planar oval billiards in \cite{bir}, applied to our billiards states that the oval billiard map $T$ has at least two 2-periodic orbits and at least four $n$-periodic orbits, for each fixed $n\neq 2$.
Generic $C^1$ planar billiards have only a finite number of nondegenerate periodic orbits, for each fixed period, as proved by Dias Carneiro, Oliffson Kamphorst and Pinto-de-Carvalho in \cite{dia}.
In this new context we get a less general result and show that having only a finite number of nondegenerate periodic orbits, for each fixed period, is an open and dense property for $C^\infty$ oval billiards on the Hyperbolic Plane and is only open on a hemisphere of the unit sphere.
 We finish this paper studying the stability of these nondegenerate orbits using the MacKay-Meiss Criterium \cite{mac}.

Billiards on the Euclidean plane were, and still are, extensively studied. Billiards on surfaces of constant curvature are much less studied and the papers focus on special properties.
For instance,  Veselov \cite{ve} , Bolotin \cite{bol}, Dragov\'ic, Jovanov\'ic and Radnov\'ic \cite{dra}, Popov and Topalov \cite{pop}, \cite{popo}  and Bialy \cite{bia} deal with the question of integrability. B.Gutkin, Smilansky and E.Gutkin \cite{gu} looked at hyperbolic billiards on the sphere and the hyperbolic plane. E.Gutkin and Tabachnikov \cite{gut} studied geodesic polygonal billiards. Blumen, Kim, Nance and Zarnitsky \cite{blu} studied periodic orbits of billiards on surfaces of constant curvature, using the tools of geometric optics.
Among them, only Bialy \cite{bia} and Zhang \cite{zh} looked more closely to oval billiards.

\section{Ovals on surfaces of constant curvature}

For the study of  billiards, we will only be interested in the behavior of the geodesics and the measure of angles. Excluding the Euclidean plane, we can then take as model of surface of constant curvature,  denoted by $S$, one of the surfaces:
 an open hemisphere of the unit sphere $\S$, given in $\R^3$ by $\{z=\sqrt{1-x^2-y^2}, z>0\}$ or
 the upper sheet of the hyperbolic plane $\H$, given in $\R^{2,1}$ by $\{z=\sqrt{1+x^2+y^2}\}$.

The geodesics on $S$ are the intersections of the surface with the planes passing by the origin.  $S$ is geodesically convex and the distance between two points $X$ and $Y$ on $S$ is measured by
$$d_S(X,Y)=\, \left\{\begin{array}{ccc}
   \arccos\langle X,Y\rangle&  \mbox{if}& X,Y\in\S\\
   \mathrm{arccosh}(-\langle\langle X,Y\rangle\rangle) & \mbox{if}& X,Y\in \H
\end{array}\right. $$
where $\langle,\rangle$ is the usual inner product on $\R^3$ and $\langle\langle,\rangle\rangle$ is the inner product on $\R^{2,1}$.

Given $X, Y \in S$, the geodesic from $X$ to $Y$ is
 $${Y\, =\, \left\{\begin{array}{rcc} X \cos d + \tau \sin d    & \mbox{  in  } & \S\\[1ex] X \cosh d + \tau \sinh d  & \mbox{  in  } & \H\\\end{array}\right.}$$ 
 where $d:=d_S(X,Y)$  and $\tau$ is the unitary tangent vector to the geodesic at $X$. 

\begin{definition} A regular curve $\Gamma(t)\subset S$ is said to be geodesically strictly convex if the intersection of  any geodesic tangent to $\Gamma$ with the curve $\Gamma$ has only one point.\end{definition}

It is proved  on \cite{litlle}  for the spherical case and  on \cite{convexgeo} for the hyperbolic case that

\begin{lemma}\label{convexcurve}
If a curve $\Gamma\subset S$ is closed, regular, simple, $C^j, j\geq2$ and has strictly positive geodesic curvature then $\Gamma$ is geodesically strictly convex.
\end{lemma}

\begin{definition} 
An oval is a regular, simple, closed, oriented, $C^j$ curve, $j\geq 2$, with  strictly positive  geodesic curvature . 
 \end{definition}

By lemma \ref{convexcurve}, any oval is geodesically strictly convex. 

\section{Billiards on ovals}

Let $\Gamma\subset S$ be an oval and $\Omega$ the region bounded by $\Gamma$.
Analogously to the planar case, we can define the billiard on $\Gamma$ as the free motion of a point particle inside $\Omega$, reflecting elastically at the impacts with $\Gamma$.
Since the motion is free, the particle moves along a geodesic of $S$ while it stays inside $\Omega$ and reflects making equal angles with the tangent at the impact with $\Gamma$. The trajectory of the particle is a geodesic polygonal line, with vertices at the impact points.

As $\Omega$ is a bounded subset of a geodesically convex surface, with strictly geodesically convex boundary, the motion is completely determined  by the impact point and the direction of movement immediately after each reflection. 
So, a parameter which locates the point of impact, and the angle between the direction of motion and the tangent to the boundary at the impact point, may be used to describe the system.

Let $l$ be the length of $\Gamma$, $s$ the arclength parameter  for $\Gamma$ and 
$\psi\in (0,\pi)$  be the angle that measures the direction of motion at the impact point. Let ${\cal C}$ be the cylinder $\R/l \mathbb Z \times(0,\pi)$.

We can define the billiard map $T_\Gamma$ on ${\cal C}$ which associates to each impact point and direction of motion $(s_0,\psi_0)$ the next impact and direction $(s_1,\psi_1)=T_\Gamma(s_0,\psi_0)$.

This billiard map defines a 2-dimensional dynamical system and the orbit of any initial point $(s_0,\psi_0)$ is the set
${\cal{O}}(s_0,\psi_0)=\{T_\Gamma^i(s_0,\psi_0)=(s_i,\psi_i), i\in\mathbb Z\}.$

\subsection{Properties of the generating function}

As above, let $d_S$ be the geodesic distance on $S$ and $\Gamma \subset S$  be an oval, parameterized by the arclength parameter $s$.

\begin{lemma}\label{lema:geradora}
Let $T_\Gamma(s_0,\psi_0)=(s_1,\psi_1)$ be the billiard map on $\Gamma$ and $g(s_0,s_1)=- d_S(\Gamma(s_0),\Gamma(s_1))$. Then
$g$ verifies
$$\frac{\partial g}{\partial s_0}(s_0,s_1)=\cos \psi_0 \quad\quad\mbox{and}\quad\quad \frac{\partial g}{\partial s_1}(s_0,s_1)=-\cos \psi_1$$
\end{lemma}

\begin{proof}
Let  $\tau_i, i=0$ or $1$, be the unitary tangent vector to the oriented geodesic joining  $\Gamma(s_0)$ to $\Gamma(s_1)$, at $\Gamma(s_i)$.

When $S=\S$ we have that 
$\cos g(s_0,s_1)=\langle\Gamma(s_0),\Gamma(s_1)\rangle$ and then
\begin{eqnarray}
\frac{\partial g}{\partial s_0}(s_0,s_1)&=&\frac{\langle\Gamma'(s_0),-\Gamma(s_1)\rangle}{\sin g(s_0,s_1)}
= \frac{\langle\Gamma'(s_0), \sin g(s_0,s_1)\, \tau_0-\cos g(s_0,s_1)\, \Gamma(s_0)\rangle}{\sin g(s_0,s_1)}\nonumber
\\
&=&\langle\Gamma'(s_0),\tau_0\rangle=\cos\psi_0.\label{eq:S1}
\end{eqnarray}
Analogously $\frac{\partial g}{\partial s_1}(s_0,s_1)=\langle\Gamma'(s_1),-\tau_1\rangle=-\cos\psi_1.$

When $S=\H$ we have that 
$\cosh g(s_0,s_1)=-\langle\langle\Gamma(s_0),\Gamma(s_1)\rangle\rangle$ and then
\begin{eqnarray}
\frac{\partial g}{\partial s_0}(s_0,s_1)&=&\frac{\langle\langle\Gamma'(s_0),-\Gamma(s_1)\rangle\rangle}{\sinh g(s_0,s_1)}
= \frac{\langle\langle\Gamma'(s_0),\sinh g(s_0,s_1)\, \tau_0-  \cosh g(s_0,s_1)\, \Gamma(s_0)\rangle\rangle}{\sinh g(s_0,s_1)}\nonumber\\
&=&\langle\langle\Gamma'(s_0), \tau_0\rangle\rangle=\cos\psi_0.\label{eq:H1}
\end{eqnarray}
Analogously $\frac{\partial g}{\partial s_1}(s_0,s_1)=\langle\langle\Gamma'(s_1),-\tau_1\rangle\rangle=-\cos\psi_1.$
 \end{proof}
 
 Let $p=-\cos\psi\in (-1,1)$. Lemma \ref{lema:geradora} implies that the arclength $s$ and the tangent momentum $p$ are conjugated coordinates with generating function $g$ for the billiard map, or, 
 $$T_\Gamma(s_0,p_0)=(s_1,p_1) \Longleftrightarrow \frac{\partial g}{\partial s_0}=-p_0 ,\,\,\, \frac{\partial g}{\partial s_1}=p_1$$
 leading to the variational definition of billiards.
 
\begin{lemma}\label{lema:segunda}
Let $k_i$ be the geodesic curvature of $\Gamma$ at $s_i, i=0,1$. The second derivatives of $g$ are
\begin{eqnarray*}
\frac{\partial^2 g}{\partial s_i^2 } (s_0,s_1)&=& \left\{
\begin{array}{lcr}
   \frac{\sin^2\psi_i}{\tan g(s_0,s_1)} + k_i\sin\psi_i &\mbox{ in }& \S\\
  \frac{\sin^2\psi_i}{\tanh g(s_0,s_1)} + k_i\sin\psi_i&\mbox{ in }& \H
\end{array}
\right.\\
& &\\
\frac{\partial^2 g}{\partial s_0 \partial s_1}(s_0,s_1) &=& \left\{
\begin{array}{lcr}
   \frac{\sin\psi_0\sin{\psi_1}}{\sin g(s_0,s_1)} & \mbox{ in }& \S\\
  \frac{\sin\psi_0\sin{\psi_1}}{\sinh g(s_0,s_1)}  &\mbox{ in }& \H
\end{array}
\right.
\end{eqnarray*}
\end{lemma}

\begin{proof}
Let  $\tau_i$ , $\eta_i$ and $\nu_i$ be the unitary tangent, normal and binormal vectors, respectively,  to the oriented geodesic joining  $\Gamma(s_0)$ to $\Gamma(s_1)$, at $\Gamma(s_i)$, seen as a curve in $\mathbb R^3$. When $S=\S$ or $\H$, since the geodesic is contained on a plane passing by the origin, $\eta_i=-\Gamma(s_i)$, $\nu_i$ is a constant unitary vector, normal to the plane, and $\{\tau_i,\nu_i\}$ is an orthonormal  basis for the tangent plane of $S$ at 
$\Gamma(s_i)$. For simplicity of notation we write just $g$ for $g(s_0,s_1)$.

 In the case $S=\S$ we differentiate (\ref{eq:S1}) getting
\begin{eqnarray*}
\frac{\partial^2 g}{\partial s_0^2}&=&\frac{-\langle\Gamma''(s_0), \Gamma(s_1)\rangle-\cos^2\psi_0\cos g}{\sin g}\\
&=&\frac{-\langle-\Gamma(s_0)+k_0\,\Gamma(s_0)\times\Gamma'(s_0), \cos g\,\Gamma(s_0)-\sin g\,\tau_0\rangle-\cos^2\psi_0\cos g}{\sin g}\\
&=& \frac{\cos g+k_0\sin\psi_0\sin g-\cos^2\psi_0\cos g}{\sin g}=\frac{\cos g\sin^2\psi_0}{\sin g} + k_0\sin\psi_0
\end{eqnarray*}
and 
\begin{eqnarray*}
\frac{\partial^2 g}{\partial s_0 \partial s_1}&=&\frac{-\langle\Gamma'(s_0), \Gamma'(s_1)\rangle+\cos\psi_0\cos\psi_1\cos g}{\sin g}\\
&=&\frac{-\langle\cos\psi_0\tau_0-\sin\psi_0\nu_0, \cos\psi_1\tau_1+\sin\psi_1\nu_1\rangle+\cos\psi_0\cos\psi_1\cos g}{\sin g}\\
&=&\frac{\sin\psi_0\sin\psi_1}{\sin g}
\end{eqnarray*}

When $S=\H$ we differentiate (\ref{eq:H1}) getting
\begin{eqnarray*}
\frac{\partial^2 g}{\partial s_0 \partial s_1}&=&\frac{-\langle\langle\Gamma'(s_0), \Gamma'(s_1)\rangle\rangle+\cos\psi_0\cos\psi_1\cosh g}{\sinh g}\\
&=&\frac{-\langle\langle-\Gamma(s_0)+k_0\,\Gamma(s_0)\times\Gamma'(s_0), \cosh g\,\Gamma(s_0)-\sinh g\,\tau_0\rangle\rangle-\cos^2\psi_0\cosh g}{\sinh g}\\
&=& \frac{\cosh g\sin^2\psi_0}{\sinh g} + k_0\sin\psi_0
\end{eqnarray*}
and 
\begin{eqnarray*}
\frac{\partial^2 g}{\partial s_0 \partial s_1}&=&\frac{-\langle\langle\Gamma'(s_0), \Gamma'(s_1)\rangle\rangle+\cos\psi_0\cos\psi_1\cosh g}{\sinh g}\\
&=&\frac{-\langle\langle\cos\psi_0\tau_0-\sin\psi_0\nu_0, \cos\psi_1\tau_1+\sin\psi_1\nu_1\rangle\rangle+\cos\psi_0\cos\psi_1\cosh g}{\sinh g}\\
&=&\frac{\sin\psi_0\sin\psi_1}{\sinh g}
\end{eqnarray*}

In both cases, the calculation of $\frac{\partial^2 g}{\partial s_1^2 }$ is analogous to $\frac{\partial^2 g}{\partial s_0^2 }$.\end{proof}

\subsection {Properties of the billiard map}

 In this subsection we will prove that
 
\begin{theorem}
 Let $\Gamma \subset S$  be a $C^j$ oval, $j\geq 2$.  The associated billiard map $T_\Gamma$ is a reversible, conservative, Twist,  $C^{j-1}$-diffeomorphism.
\end{theorem}

The proof will follow directly from the lemmas below. As above, $s$ stands for the arclength parameter for $\Gamma$, $k_i$ is the geodesic curvature of $\Gamma$ at $s_i$, $d_S$ is the distance on $S$ and  $g(s_0,s_1)=- d_S(\Gamma(s_0),\Gamma(s_1))$.

\begin{lemma}
$T_\Gamma$ is invertible and reversible.
\end{lemma}
\begin{proof} Any trajectory of the billiard problem can be travelled in both senses.  So, if $T_\Gamma(s_i,\psi_i)=(s_{i+1},\psi_{i+1})$ then $T_\Gamma^{-1}(s_i,\pi-\psi_i)=(s_{i-1},\pi-\psi_{i-1})$.

 Let $I$ be the involution on ${\cal C}$ given by $I(s,\psi)=(s,\pi-\psi)$. Clearly $I^{-1}=I$.
 
 We have then that $T_\Gamma^{-1}=I\circ T_\Gamma\circ I$ or $I\circ T_\Gamma^{-1}=T_\Gamma\circ I$, i.e., $T_\Gamma$ is reversible. \end{proof}
 
 \begin{lemma}\label{lemma:difeo}
 $T_\Gamma$ is a $C^{j-1}$ diffeomorphism.
 \end{lemma} 
 \begin{proof}
 Let $T_\Gamma(\overline s_0,\overline \psi_0)=(\overline s_1,\overline \psi_1)$ and $V_0$ and $V_1$ be two disjoint open intervals containing $\overline s_0$ and $\overline s_1$ respectively. We define 
 $$G:V_0\times V_1\times (0,\pi)\mapsto \mathbb R, \quad G(s_0,s_1,\psi_0)=\frac{\partial g}{\partial s_0}(s_0,s_1)-\cos\psi_0.$$
 
$G$ is a $C^{j-1}$ function, since $\Gamma$ and $g$ are $C^j$.  Then, by lemma \ref{lema:geradora}, $G(\overline s_0, \overline s_1,\overline \psi_0)=0$  and $\frac{\partial G}{\partial s_1}(s_0,s_1)=\frac{\partial^2 g}{\partial s_0 \partial s_1}(\overline s_0,\overline s_1)\neq 0$ by lemma \ref{lema:segunda}, since $\overline \psi_0,\overline \psi_1\in (0,\pi)$. So we can locally define a $C^{j-1}$ function $s_1=s_1(s_0,\psi_0)$ such that $G(s_0,s_1(s_0,\psi_0),\psi_0)=0$.

Taking now $\psi_1(s_0,\psi_0)=\arccos(-\frac{\partial g}{\partial s_1}(s_0,s_1(s_0,\psi_0)))$ we conclude that $T_\Gamma(s_0,\psi_0)=(s_1(s_0,\psi_0),\psi_1(s_0,\psi_0))$ is a $C^{j-1}$ function. 

As $T_\Gamma$ is invertible, with $T_\Gamma^{-1}=I\circ T_\Gamma\circ I$, we conclude that $T_\Gamma$ is a $C^{j-1}$ diffeomorphism. \end{proof}

Differentiating the expressions
$\cos\psi_0=\frac{\partial g}{\partial s_0}(s_0,s_1(s_0,\psi_0))$ and 
$\cos\psi_1(s_0,\psi_0)=-\frac{\partial g}{\partial s_1}(s_0,s_1(s_0,\psi_0))$
we obtain
\begin{lemma}\label{lema:formulas}
\begin{eqnarray*}\label{eq:twist}
\frac{\partial^2g}{\partial s_0^2}+\frac{\partial^2g}{\partial s_0\partial s_1}\frac{\partial s_1}{\partial s_0}=0 &\quad\quad\quad\quad\quad\quad&
\frac{\partial^2g}{\partial s_0\partial s_1}\frac{\partial s_1}{\partial \psi_0}=-\sin\psi_0\\
\frac{\partial^2g}{\partial s_0\partial s_1}+\frac{\partial^2g}{\partial s_1^2}\frac{\partial s_1}{\partial s_0}=
\sin\psi_1\frac{\partial \psi_1}{\partial s_0}&\quad\quad\quad\quad\quad\quad\quad&
\frac{\partial^2g}{\partial s_1^2}\frac{\partial s_1}{\partial \psi_0}=\sin\psi_1\frac{\partial \psi_1}{\partial\psi_0}
\end{eqnarray*}
\end{lemma}

\begin{lemma}
$T_\Gamma$ is a Twist map.
\end{lemma}
\begin{proof}
By lemmas \ref{lema:segunda} and \ref{lema:formulas} and remembering that  $\psi\in (0,\pi)$ and $g<0$, we have that  $\frac{\partial s_1}{\partial \psi_0}>0$ and $T_\Gamma$ has the Twist property.\end{proof}

Using the formulas of lemmas \ref{lema:formulas} and \ref{lema:segunda}, we obtain the derivative of the billiard map as:
\begin{lemma}\label{lema:DT}
$DT_\Gamma(s_0,\psi_0)=\left( 
\begin{array}{cc}
 \frac{\partial s_1}{\partial s_0}&  \frac{\partial s_1}{\partial \psi_0} \\
 \frac{\partial \psi_1}{\partial s_0}&  \frac{\partial \psi_1}{\partial\psi_0}                                              
  \end{array}\right)$ where
  \begin{itemize}
   \item in $\S$
  $$
  \frac{\partial s_1}{\partial s_0}= \frac{-k_0\sin g-\sin\psi_0\cos g}{\sin\psi_1}, \ \ 
   \frac{\partial s_1}{\partial \psi_0} =\frac{-\sin g}{\sin\psi_1}, \ \  \frac{\partial \psi_1}{\partial\psi_0} =  \frac{-\sin\psi_1\cos g-k_1\sin g}{\sin\psi_1}$$
 $$\frac{\partial \psi_1}{\partial s_0}= \frac{-k_0\sin\psi_1\cos g+\sin\psi_0\sin\psi_1\sin g-k_0k_1\sin g-k_1\sin\psi_0\cos g}{\sin\psi_1}$$
   \item in $\H$
 $$
  \frac{\partial s_1}{\partial s_0}= \frac{-k_0\sinh g-\sin\psi_0\cosh g}{\sin\psi_1},\ \ 
   \frac{\partial s_1}{\partial \psi_0} = -\frac{\sinh g}{\sin\psi_1},\ \  \frac{\partial \psi_1}{\partial\psi_0} =  \frac{-\sin\psi_1\cosh g-k_1\sinh g}{\sin\psi_1}$$
$$ \frac{\partial \psi_1}{\partial s_0}= \frac{-k_0\sin\psi_1\cosh g-\sin\psi_0\sin\psi_1\sinh g-k_0k_1\sinh g-k_1\sin\psi_0\cosh g}{\sin\psi_1}$$
  \end{itemize}
\end{lemma}

Calculating now the determinant of $DT_\Gamma$ it is easy to see that
\begin{lemma}
$T_\Gamma$ preserves the measure $d\mu=\sin\psi ds d\psi$.
\end{lemma}

\section{Periodic Orbits}

The oval billiard map $T_\Gamma$ is then a conservative reversible discrete 2-dimensional dynamical system, defined by a $C^{j-1}$-Twist map, $j\geq 2$. 
To each $n$-periodic orbit of $T_\Gamma$ is associated a closed geodesic polygon with vertices on the oval $\Gamma$. Among them we distinguish the Birkhoff periodic orbits of type $(m,n)$, the $n$-periodic orbits such that the corresponding trajectory winds $m$ times around $\Gamma$ before closing, meaning that the orbit has rotation number $m/n$.
The classical result for Birkhoff periodic orbits of Twist maps (see, for instance, \cite{kat}, page 356)  applied to our billiards\footnote{ For oval planar billiards, this result was proven by Birkhoff in \cite{bir}. } states that:

\begin{theorem}
 Given relatively primes $m$ and $n$,  $n\geq2$ and $0<m<n$, there exist at least two Birkhoff orbits of type $(m,n)$ for the oval billiard map $T_\Gamma$.
 \end{theorem}

{\em At least two} does not necessarily mean {\em in a finite number}.  As was proved by Bolotin in \cite{bol}, the geodesic circular billiard on $S$ is integrable and then has infinitely many Birkhoff orbits of any period. On the other side, generic $C^1$-diffeomorphisms defined on compact sets have a finite number of nondegenerate periodic orbits of each period. This will be the case also for oval billiards on $S$, although encountering here two main differences: the domain of the billiard map is an open cylinder and perturbations of billiards as diffeomorphisms may not be billiards.

 To circumvent these problems we will perturb the boundary curve $\Gamma$, instead of the map itself, and will find compact sets on the open cylinder where the periodic orbits lay,  analogously to the way it was done by Dias Carneiro, Oliffson Kamphorst and Pinto-de-Carvalho  \cite{dia} for planar billiards.  For technical reasons that will be clear below, we will only consider $C^\infty$ boundary ovals. Using those facts we will be able to prove that 

\begin{theorem}
 For each fixed period $n\geq 2$, having only a finite number of $n$-periodic orbits, all nondegenerate, is an open and dense property for $C^\infty$ oval billiards on $\H$. For $C^\infty$ oval billiards on $\S$ it is only an open property.
\end{theorem}

\subsection{Normal perturbations of ovals}

Let $\Gamma:I\mapsto S$ be a $C^\infty$ oval parameterized by the arclength parameter $s$ and $\eta(s)$ be the unitary normal vector such that $\{\Gamma'(s),\eta(s)\}$ is an oriented positively and orthonormal basis of the tangent plane of $S$ at $\Gamma(s)$.

\begin{definition}
$\beta$ is $\epsilon$-$C^2$-close to $\Gamma$ if $\beta$ can be written as
 \begin{equation}\label{pertnormal}
{\beta(s)\, =\,\left\{\begin{array}{lcl}
 \displaystyle{\frac{\Gamma(s)\,+\,\lambda(s)\eta(s)}{\sqrt{1+\lambda^2(s)}}}&\mbox{in}&\S\\[1ex]
\displaystyle{\frac{\Gamma(s)\,+\,\lambda(s)\eta(s)}{\sqrt{1-\lambda^2(s)}}}&\mbox{in}&\H
\end{array}\right.}
\end{equation}
where $\lambda:I\to\R$ is $C^j, j\geq2$ with $||\lambda||_2 <\epsilon$.
\end{definition}

Remark that if $\beta$ is $\epsilon$-$C^2$-close to $\Gamma$ then the trace of $\beta$ is contained on the  tubular neighbourhood $V_{\epsilon}(\Gamma)$, given by the radial projection of the set $\{\Gamma(s)+\lambda\eta(s),\, s\in I,\, -\epsilon<\lambda<\epsilon \} $ onto $S$, which, as $\Gamma$ is an oval, is an open subset of $S$ for $\epsilon$ sufficiently small.

\begin{lemma}
If $\epsilon$ is sufficiently small, $\lambda$ is $C^\infty$ and $\beta$ is $\epsilon$-$C^2$-close to $\Gamma$, then $\beta$ is a $C^\infty$ oval.
\end{lemma}

\begin{proof} As $\lambda$ is a $C^\infty$ function, $\beta$ is a $C^\infty$ curve. As $\Gamma$ is closed, $\beta$ is a closed curve.

Moreover, $\beta'(s)=\Gamma'(s)+R_1(s,\lambda(s),\lambda'(s))$ and $\beta''(s)=\Gamma''(s)+R_2(s,\lambda(s),\lambda'(s), \lambda''(s))$ with $||R_i||\to 0$ if $||\lambda||_2\to 0$, uniformly on $s$, which implies that, if $\epsilon$ is sufficiently small, $\beta$ is regular and has strictly positive geodesic curvature.
\end{proof}

\begin{lemma}\label{lema:gera}
If $\epsilon$ is sufficiently small and $\beta$ is $\epsilon$-$C^2$-close to $\Gamma$, then their associated generating functions $g_\Gamma$ and $g_\beta$ are close in the $C^2$ topology.
\end{lemma}
\begin{proof}
Let $\Gamma$ be a $C^\infty$ oval on $S$, parameterized by the arclength parameter $s$ and $\beta$ be a normal perturbation of $\Gamma$ as in (\ref{pertnormal}), parameterized by the arclength parameter $\sigma$. 
The generating functions are $g_\Gamma(s_0,s_1)=-d_S(\Gamma(s_0),\Gamma(s_1))$ and $g_\beta(s_0,s_1)=-d_S(\beta(\sigma(s_0)),\beta(\sigma(s_1)))$. 

The result will follow immediately from the following remarks:\\
 $\beta(s)=\Gamma(s)+R_0(s,\lambda(s))$, $\beta'(s)=\Gamma'(s)+R_1(s,\lambda(s),\lambda'(s))$,  $\beta''(s)=\Gamma''(s)+R_2(s,\lambda(s),\lambda'(s), \lambda''(s))$ and $\sigma=s+R_3(s,\lambda,\lambda')$,  with $\|R_i\|\to 0$ if $\|\lambda\|_2\to 0$, uniformly on $s$.
\end{proof}

\begin{proposition}
If $\epsilon$ is sufficiently small and $\beta$ is $\epsilon$-$C^2$-close to $\Gamma$ then their associated billiard maps $T_\Gamma$ and $T_\beta$ are close in the $C^1$ topology.
\end{proposition}
\begin{proof}
By the construction of the billiard map from the generating function, as was done for instance in the proof of lemma \ref{lemma:difeo}, billiard maps will be close in the $C^1$ topology if their generating functions are close in the $C^2$ topology. So the result follows directly from the lemma \ref{lema:gera}.
\end{proof}

\subsection{Finite number of nondegenerate $n$-periodic orbits}

\subsubsection{Openness}

\begin{lemma}\label{lema:compacto}
Let $\Gamma:I\to S$ be an oval. There exists a positive real number  $\delta_n$  such that any $n$-periodic orbit of $T_\Gamma$ has at least one point on the compact strip $I\times [\delta_n,\pi-\delta_n]$.
\end{lemma}
\begin{proof}
Let $\{(\overline s_0,\overline\psi_0),...,(\overline s_{n-1},\overline\psi_{n-1})\}$ be an $n$-periodic orbit of $T_\Gamma$. Then the points $\Gamma(\overline s_i)$ are the vertices of a geodesic polygon $P$ inscribed on $\Gamma$. Let us suppose that this geodesic polygon is simple and let $\zeta_i$ be the internal angles.

In $\H$, by Gauss-Bonnet Theorem, $\sum \zeta_i\leq (n-2)\pi$. Then $\sum\overline\psi_i\geq \pi$ and there exists $i_0$ such that $\overline \psi_{i_0}\geq\frac{\pi}{n}$. By the reversibility of $T_\Gamma$, $\pi-\overline\psi_{i_0}\geq\frac{\pi}{n}$ and then $\frac{\pi}{n}\leq\overline\psi_{i_0}\leq\pi-\frac{\pi}{n}$. 

In $\S$, Gauss-Bonnet Theorem gives $\sum \zeta_i> (n-2)\pi$.  So we have to work in a slightly different way. Since $\Gamma(I)\subset \S$ there exists  $m_0>n$ such that the area $A_\Gamma$ enclosed by $\Gamma$ satisfies $A_\Gamma<2\pi-\frac{\pi n}{m_0}$.
Suppose now that there is an $n$-periodic orbit $\{(\overline s_0,\overline\psi_0),...,(\overline s_{n-1},\overline\psi_{n-1})\}$, associated to a simple geodesic polygon $P$ and such that $\overline\psi_i<\frac{\pi}{m_0}$ for all $i$. The area $A_P$ enclosed by $P$ satisfies $A_P<A_\Gamma<2\pi-\frac{\pi n}{m_0}$. But, again by Gauss-Bonnet Theorem, 
$A_P\geq2\pi-\frac{1}{2}\sum(\pi-\zeta_i)=2\pi-\sum\overline\psi_i>2\pi-\frac{\pi n}{m_0}>A_\Gamma$
and then there is at least one $i_0$ such that $\frac{\pi}{m_0}<\overline\psi_i$. Once more, by the reversibility of $T_\Gamma$,  $\frac{\pi}{m_0}\leq\overline\psi_{i_0}\leq\pi-\frac{\pi}{m_0}$. 

If the geodesic polygon $P$ is not simple, we can take a new simple polygon $\tilde P$ with the same vertices on $\Gamma$ as $P$ and with internal angles $\tilde \zeta_i$. Taking only the internal angles $\zeta_i$ of $P$ at the vertices on $\Gamma$ we have that $\sum\zeta_i<\sum\tilde\zeta_i$ and the result follows. \end{proof}

\begin{proposition} \label{prop:open}
For a fixed period $n\geq 2$, the set of $C^\infty$-ovals on $S$ such that its associated billiard map has only a finite number of $n$-periodic orbits, all nondegenerate, is an open set.
\end{proposition}
\begin{proof} Suppose that the $C^\infty$ diffeomorphism $T_\Gamma$ has only nondegenerate $n$-periodic orbits. To each one of these orbits corresponds a fixed point of $T_\Gamma^n$ on the compact strip $I\times [\delta_n,\pi-\delta_n]$. So they must be in a finite number and then $T_\Gamma$ has only a finite number of nondegenerate $n$-periodic orbits. 
Taking $\epsilon$ sufficiently small, any perturbation $\beta$ $\epsilon$-$C^2$-close to $\Gamma$ corresponds to a billiard map $T_\beta$ $C^1$-close to $T_\Gamma$ and will have also only a finite number of nondegenerate $n$-periodic orbits. \end{proof}

\subsubsection{Density}

Suppose $\{(\overline s_0,\overline\psi_0), (\overline s_1,\overline\psi_1),..., (\overline s_{n-1},\overline\psi_{n-1})\}$ a degenerate $n$-periodic orbit for $T_\Gamma$.
As det($DT_\Gamma^n|_{(\overline s_0,\overline\psi_0)}) =1$,  being degenerate translates as $\mbox{tr}(DT_\gamma^n|_{(\overline s_0,\overline\psi_0)} )=\pm 2.$

By  lemma \ref{lema:DT}
\begin{eqnarray*}
DT^n_{(\overline s_0,\overline\psi_0)}& =&DT_{(\overline s_{n-1},\overline\psi_{n-1})}DT_{(\overline s_{n-2},\overline\psi_{n-2})}...DT_{(\overline s_0,\overline\psi_0)}\\
&=&\frac{1}{\sin\overline\psi_0...\sin\overline\psi_{n-1}}A_{n-1}....A_1A_0
\end{eqnarray*}
Each matrix $A_i=k_{i+1}k_iB_i+k_{i+1}C_i+k_iD_i+E_i$, where $k_i$  is the geodesic curvature of  $\Gamma$ at $s_i$, and the entries of the matrices $B_i, C_i, D_i, E_i$ depend only on the angles $\overline\psi_i$ and $\overline\psi_{i+1}$ and on the geodesic distance between $\Gamma(s_i)$ and $\Gamma(s_{i+i})$. 
 
Let us fix our attention on one impact point of the degenerate $n$-periodic trajectory, say $\Gamma(s_1)$. If it impacts $m_1$ times at  $\Gamma(s_1)$  then 
$$\mbox{tr}(DT^n_{(\overline s_0,\overline\psi_0)})=p_1(k_1)=b_{m_1}k_1^{m_1}+...+b_1k_1+c_1=\pm 2$$
where the coefficients $b_j$ and $c_1$  do not depend on $k_1$.
  
If any of the $b_j\neq 0$, we can take a sufficiently small normal perturbation $\beta$, as in (\ref{pertnormal}), with a function $\lambda$ satisfying  $\lambda(s_1)=0, \lambda'(s_1)=0, \lambda''(s_1)\neq 0$, and $\lambda\equiv 0$ outside an interval containing $s_1$ and no other point of the trajectory. So we preserve the periodic orbit but change the geodesic curvature at the vertex $\Gamma(s_1)$, which implies 
$\mbox{tr}(DT^n_\beta )\neq \mbox{tr}(DT^n_\Gamma)$, i.e, on $\beta$ the orbit is nondegenerate.

If all the $b_j=0$ we take the next impact $s_2$ (as a billiard has no fixed point, $s_2\neq s_1$) and then:
$$\mbox{tr}(DT^n_{(\overline s_0,\overline\psi_0)})=p_2(k_2)=b_{m_2}k_2^{m_2}+...+b_1k_2+c_2=c_1=\pm 2$$
where the coefficients $b_j$ and $c_2$  do not depend on $k_1$ and $k_2$. 

If any of the $b_j\neq 0$ we can take the small perturbation $\beta$ at $s_2$ as above. Otherwise, we continue the process till the last impact, say $s_0$. We have then
$$\mbox{tr}(DT^n_{(\overline s_0,\overline\psi_0)})=p_0(k_0)=b_{m_0}k_0^{m_0}+...+b_1k_0+c_0=\pm 2$$
where the coefficients $b_j$ and $c_0$  do not depend any more on any of the geodesic curvatures $k_i$.

For $\H$ we have calculated $c_0$ obtaining 
 $c_0=(-1)^n2\cosh L\neq \pm 2$ 
where  $L\neq 0$ is the perimeter of the geodesic polygonal trajectory.

Then, in this case, there is a $j$ such that  $b_j\neq 0$ and we can approach $T_\Gamma$ by billiards with a nondegenerate $n$-periodic orbit.

As nondegenerate periodic orbits are isolated, with a finite number of perturbations we can construct an oval $\beta$ as close as we want to $\Gamma$ such that the associated billiard map has only a finite number of nondegenerate $n$-periodic orbits.

We have then
\begin{proposition} \label{prop:dense}
For a fixed period $n\geq2$, the set of $C^\infty$-ovals on $\H$ such that its associated billiard map has only a finite number of $n$-periodic orbits, all nondegenerate, is a dense set.
\end{proposition}

Unfortunately, the same techniques do not work at $\S$. If all the $b_j$ are zero except for the last impact $s_0$ we get 
$\mbox{tr}(DT^n_{(\overline s_0,\overline\psi_0)})=p_0(k_0)=b_{m_0}k_0^{m_0}+...+b_1k_0+c_0=\pm 2$
where, taking $L$ as the perimeter of the trajectory, 
 $c_0=(-1)^n\,2\,cos\, L$ and all the coefficients $b_j$ are multiples of  $\sin L$.
So, if  $L=\mu\pi$, $c_0=\pm 2$ and all the $b_j=0$ and our normal perturbation do not destroy the degenerescence of the orbit. 

Summarizing we have
\begin{theorem}
For each fixed period $n\geq 2$, having only a finite number of $n$-periodic orbits, all nondegenerate, is an open and dense property for $C^\infty$ oval billiards on $\H$. For $C^\infty$ oval billiards on $\S$ it is only an open property.
\end{theorem}

\subsubsection{A remark about the planar case}

Dias Carneiro, Oliffson Kamphorst and Pinto-de-Carvalho \cite{dia} proved that having only a finite number of $n$-periodic orbits, all nondegenerate, is a generic property for $C^2$ planar oval billiards.

Although their result is true, they did not analyze the trajectories with multiple impact points. Using our techniques we can fill in this gap remarking that if all the $b_j$ are zero we get in the final step
$$\mbox{tr}(DT^n_{(\overline s_0,\overline\psi_0)})=p_0(k_0)=b_{m_0}k_0^{m_0}+...+b_1k_0+(-1)^n2L=\pm 2$$
where  $L$ the perimeter of the trajectory. 
 
If $L\neq 1$, $c_0\neq \pm 2$ and then there is a $b_j\neq 0$ and the special normal perturbation on the plane destroys the degenerescence. 

If $L=1$, the first coefficient $b_1=(-1)^{n+1}2(\frac{1}{\sin\overline\psi_0}+...+\frac{1}{\sin\overline\psi_{m_0}})\neq 0$ and, again, the special normal perturbation solves the problem.

\subsection{Stability of periodic orbits}

Let  ${\cal O}( s_0,\psi_0)=\{( s_0,\psi_0), ( s_1,\psi_1),..., (s_{n-1},\psi_{n-1})\}$ be an $n$-periodic orbit for $T_\Gamma$.

As det($DT_\Gamma^n|_{( s_0,\psi_0)}) =1$,  ${\cal O}( s_0,\psi_0)$ is hyperbolic if $|\mbox{tr}(DT_\gamma^n|_{(s_0,\psi_0)} )|> 2$, elliptic if $|\mbox{tr}(DT_\gamma^n|_{( s_0,\psi_0)} )|<2$ and parabolic if $|\mbox{tr}(DT_\gamma^n|_{( s_0,\psi_0)} )|=2$.

Let $W(s_0,s_1,...,s_{n-1})=\sum_{i=0}^{n-1} g(s_i,s_{i+1}),$ where $g(s_0,s_1)=-d_s(\Gamma(s_0),\Gamma(s_1))$ is the generating function and $s_n=s_0$, be the action defined on the $n$-torus $\mathbb R/l\mathbb Z\times...\times \mathbb R/l\mathbb Z$ minus the set $\{(s_0,s_1,...,s_{n-1}) \, \mbox{s.t}\, \exists i\neq j \,\mbox{with}\,  s_i=s_j\}$.

The critical points of $W$ are the coordinates of the vertices on $\Gamma$ of the $n$-periodic trajectories, since $\frac{\partial g}{\partial s_i}(s_{i-1},s_i)=-\frac{\partial g}{\partial s_i}(s_i,s_{i+1})$. Remark that $W<0$ and it will always have a global minimum (perhaps degenerate), but not a global maximum.

The MacKay-Meiss formula \cite{mac} relates the derivative of $T_\Gamma^n$ with the Hessian matrix $H$ of $W$ by
$$2-\mbox{tr}DT^n_\Gamma|_{(s_0,\psi_0)}=(-1)^{n+1}\frac{\det H}{b_0b_1...b_{n-1}}$$
where $b_i=\frac{\partial^2 g}{\partial s_i\partial s_{i+1}}(s_i,s_{i+1})$.

It implies that the nondegenerate critical points of the action $W$ are the nondegenerate $n$-periodic orbits of $T_\Gamma$.

So, if $T_\Gamma$ has only nondegenerate $n$-periodic orbits, the nondegenerate minima of $W$ will always be hyperbolic. The other $n$-periodic orbits can be either hyperbolic or elliptic.

\section*{Acknowledgements}
 We thank the Brazilian agencies FAPEMIG, CAPES and CNPq for financial support.

\end{document}